\documentclass[a4paper, 12pt]{amsart}
\usepackage[top=40truemm,bottom=40truemm,left=35truemm,right=35truemm]{geometry}
\usepackage{amssymb,amsmath}

\newcommand{\A}{\mathbb{A}}

\newcommand{\Z}{\mathbb{Z}}
\newcommand{\BP}{\mathbb{P}}
\newcommand{\Q}{\mathbb{Q}}

\newcommand{\G}{\mathbb{G}}

\newcommand{\red}{{\rm red}}
\newcommand{\Bk}{\operatorname{Bk}}

\newcommand{\Supp}{\operatorname{Supp}}

\newcommand{\ch}{\operatorname{char}}

\newcommand{\SO}{\mathcal{O}}
\newcommand{\SL}{\mathcal{L}}

\newcommand{\ol}[1]{\overline{#1}}

\newcommand{\lkd}[1]{\ol{\kappa}({#1})}
\newtheorem{thm}{Theorem}[section]
\newtheorem{prop}[thm]{Proposition}
\newtheorem{lem}[thm]{Lemma}

\newtheorem{defn}[thm]{Definition}
\newtheorem{example}[thm]{Example}
\newtheorem{remark}[thm]{Remark}

\newtheorem*{question}{Question}
\begin{document}
\title[Smooth affine surfaces of $\ol{\kappa} = 1$]{Logarithmic multicanonical systems of smooth affine surfaces of logarithmic Kodaira dimension one }
\author{Hideo Kojima}
\address[H. Kojima]{Department of Mathematics, Faculty of Science, Niigata University, 8050 Ikarashininocho, Nishi-ku, Niigata 950-2181, Japan}
\email{kojima@math.sc.niigata-u.ac.jp}
\date{5 April 2024}
\subjclass[2020]{Primary 14J26; Secondary 14R20}
\thanks{The author is supported by JSPS KAKENHI Grant Number JP 21K03200.}
\begin{abstract}
Let $S$ be a smooth affine surface of logarithmic Kodaira dimension one and let $(V,D)$ be a pair of a smooth projective surface $V$ and a simple normal crossing divisor $D$ on $V$ such that $V \setminus \Supp D = S$. In this paper, we consider the logarithmic multicanonical system $|m(K_V + D)|$. We prove that, for any $m \geq 8$, $|m(K_V+D)|$ gives an $\BP^1$-fibration form $V$ onto a smooth projective curve. 
\end{abstract}
\maketitle

\setcounter{section}{0}

\section{Introduction}

We work over an algebraically closed field $k$ of characteristic $\geq 0$.

Let $S$ be a smooth open algebraic surface of logarithmic Kodaira dimension one and let $(V,D)$ be a pair of a smooth projective surface $V$ and a simple normal crossing divisor $D$ on $V$ such that $S = V \setminus \Supp D$. Here we call the pair $(V,D)$ an SNC-completion of $S$. It then follows from \cite{Kaw78} in the case $\ch (k) = 0$ and \cite[Theorem 2.1]{K13} in the general case that, for a sufficiently large $n$, the complete linear system $|n(K_V+D)|$ defines a fibration $\varphi: V \to B$ onto a smooth projective curve $B$ such that $\varphi$ is an elliptic fibration, a quasi-elliptic fibration or a $\BP^1$-fibration. 
%Furthermore, if $\varphi$ is a $\BP^1$-fibration, then $F \cdot D = 2$ for a fiber $F$ of $\varphi$. 
In this article, we consider the following question, which is an open surface version of the question considered in Katsura--Ueno \cite{KU85}, Katsura \cite{Katsura95}, \cite{Katsura18}, Katsura--Saito \cite{KS21} (see also Iitaka \cite{I70}).

%Question
\begin{question}
Let the notations and assumptions be the same as above.
\begin{itemize}
\item[(1)]
Does there exist a positive integer $M$ such that, if $m \geq M$, then the logarithmic multicanonical system $|m(K_V+D)|$ gives a structure of an elliptic fibration, a quasi-elliptic fibration or a $\BP^1$-fibration for any open algebraic surface $S$ of $\lkd{S}  = 1$ with an SNC-completion $(V,D)$?
\item[(2)]
What is the smallest $M$ which satisfies the property of {\rm (1)}?
\end{itemize}
\end{question}

In the present article, we prove that, for the smooth affine surfaces of logarithmic Kodaira dimension one, $M = 8$ and $8$ is best possible. 
The following result is the main theorem of this article.

%Theorem 1.1
\begin{thm}
Let $S$ be a smooth affine surface of logarithmic Kodaira dimension one and let $(V,D)$ be its SNC-completion. 
Let $(K_V+D)^{+}$ be the nef part of the Zariski decomposition of $K_V+D$. 
Then, for every integer $m \geq 8$, $| \lfloor m (K_V+D)^+ \rfloor | $ induces a $\BP^1$-fibration onto a smooth projective curve. Furthermore, the number $8$ is best possible.
\end{thm}

In the setting of Theorem 1.1, if $|\lfloor m(K_V + D)^+ \rfloor |$ induces a $\BP^1$-fibration, then so does $|\lfloor m(K_{\tilde{V}} + \tilde{D})^+ \rfloor |$ for any SNC-completion $(\tilde{V}, \tilde{D})$ of $S$. 
Let $F$ be a fiber of the $\BP^1$-fibration. Then $F^2 = F \cdot (K_V + D) = 0$. So $F \cdot D = 2$. 

In Section 2, we summarize some results on smooth affine surfaces. In particular, we give a structure theorem for smooth affine surfaces of $\ol{\kappa} = 1$ (see Theorem 2.7), which is the key result in the proof of Theorem 1.1. In Section 3, we prove Theorem 1.1. We first prove that, for every integer $m \geq 8$, $| \lfloor m (K_V+D)^+ \rfloor | $ induces a $\BP^1$-fibration onto a smooth projective curve. Then we construct a smooth affine surface $S'$ of $\lkd{S'} = 1$ such that $|\lfloor 7(K_{V'}+ D')^+ \rfloor |$, where $(V', D')$ is an SNC-completion of $S'$, does not induce a $\BP^1$-fibration. Finally, in Section 4, we give some remarks for special smooth affine surfaces of $\ol{\kappa} = 1$.

\section{Preliminaries}

We employ the following notations. For the definitions of the logarithmic $m$-genus and the logarithmic Kodaira dimension of a smooth algebraic variety, see \cite{OAS}.
\smallskip

$\lkd{S}$: the logarithmic Kodaira dimension of $S$.

$\ol{P}_m(S)$: the logarithmic $m$-genus of $S$.

$\mu^*(D)$: the total transform of $D$ by $\mu$. 

$\mu_*^{-1}(D)$: the proper transform of $D$ by $\mu$. 

$\lfloor Q \rfloor$: the integral part of a $\Q$-divisor $Q$.

$\lceil Q \rceil:= - \lfloor -Q \rfloor$: the roundup of a $\Q$-divisor $Q$.

\medskip

From now on, we recall the structure theorem for open algebraic surfaces of logarithmic Kodaira dimension one in \cite[Sections 1 and 2]{K13}. Since we consider smooth affine surfaces only, some of the notions in \cite[Sections 1 and 2]{K13} are not necessary. So we shorten the arguments and results there. 

Let $T$ be an SNC-divisor on $V$. Then $T$ is called a {\em chain} if its dual graph is linear, and in writing it as a sum of irreducible components $T = T_1 + \cdots + T_r$, we assume that $T_i \cdot T_{i+1} = 1$ for $1 \leq i \leq r-1$. We define the type of the chain $T$ as the sequence $[- T_1^2, - T_2^2, \ldots, -T_r^2]$. We often write as $T = [- T_1^2, - T_2^2, \ldots, -T_r^2]$.

We recall some basic notions in the theory of peeling. For more details, see \cite[Chapter 2]{OAS} or \cite[\S 1]{MT84}. 
Let $(V,D)$ be a pair of a smooth projective surface $V$ and an SNC-divisor $D$ on $V$. We call such a pair $(V,D)$ an {\em SNC-pair}. 
A connected curve in $D$ means a connected curve consisting of irreducible components of $D$. A connected curve $T$ in $D$ is a {\em twig} if the dual graph of $T$ is a chain and $T$ meets $D-T$ in a single point at one of the end components of $T$, the other end of $T$ is called the {\em tip} of $T$. 
A connected curve $E$ in $D$ is {\em rational} (resp.\ {\em admissible}) if it consists only of rational curves (resp.\ if there are no $(-1)$-curves in $\Supp E$ and the intersection matrix of $E$ is negative definite). An admissible rational twig in $D$ is {\em maximal} if it is not extended to an admissible rational twig with more irreducible components of $D$. 

From now on, we assume that the surface $S:= V \setminus \Supp D$ is affine. 
Then $S$ contains no complete algebraic curves, $\Supp D$ is connected, and the divisor $D$ is big. So the divisor $D$ does not contain neither an admissible rational rod nor an admissible rational fork (see \cite[Chapter 2]{OAS} or \cite[Section 1]{K13} for the definitions of a rod and a fork). Furthermore, in the process of construction of a strongly minimal model of $(V,D)$ given in \cite[Sections 1 and 2]{K13}, neither an admissible rational rod nor an admissible rational fork does not appear. 
We note that every maximal rational twig in $D$ is admissible provided $\lkd{S} \geq 0$. 

A $(-1)$-curve $E$ on $V$ is called a {\em superfluous exceptional component} of $D$ if $E \subset \Supp D$, $E \cdot (D-E) \leq 2$ and the equality holds if and only if $E$ meets two irreducible components of $D - E$. If $E$ is an superfluous exponential component of $D$, let $f: V \to V'$ be the contraction of $E$ and set $D' := f_*(D)$. Then $V' \setminus \Supp D' = V\setminus \Supp D$ and $D'$ is an SNC-divisor. 
So when we consider an SNC-pair $(V,D)$, we may assume that $D$ has no superfluous exponential components. 

Let $\{ T_{\lambda} \}$ be the set of all maximal admissible rational twigs. Then there exists a unique decomposition of $D$ as a sum of effective $\Q$-divisors $D= D^{\#} + \Bk(D)$ such that the following two conditions (i) and (ii) are satisfied:
\begin{enumerate}
\item[(i)]
$\Supp (\Bk(D)) =  \cup_{\lambda} T_{\lambda}$.
\item[(ii)]
$(K_V + D^{\#}) \cdot Z = 0$ for every irreducible component $Z$ of $\Supp (\Bk(D))$.
\end{enumerate}
The $\Q$-divisor $\Bk(D)$ is called the {\em bark} of $D$. We note that $\Supp (\Bk(D)) = \Supp (D - \lfloor D^{\#} \rfloor)$. Furthermore, $\lceil D^{\#} \rceil = D$ since $S$ is affine. 

%Definition 2.1
\begin{defn}
{\rm Let $(V,D)$ be the same as above, where $S = V \setminus \Supp D$ is affine. Then $(V,D)$ is said to be {\em almost minimal} if, for every irreducible curve $U$ on $V$, either $U \cdot (K_V + D^{\#}) \geq 0$ or $U \cdot (K_V + D^{\#}) < 0$ and the intersection matrix of $U+\Bk(D)$ is not negative definite.}
\end{defn}

%Lemma 2.2
\begin{lem}
With the same notations and assumptions as in Definition {\rm 2.1}, assume that $(V,D)$ is not almost minimal and there exist no superfluous exponential components of $D$. Then the following assertions hold true{\rm :}
\begin{itemize}
\item[{\rm (1)}]
There exists a $(-1)$-curve $U$ such that $U \not\subset \Supp D$, $U \cdot D = 1$, $U$ meets only one maximal admissible rational twig, say $T_1$ of $D$. In particular, $V \setminus \Supp (U + D)$ is an affine open subset of $S$. 
\item[{\rm (2)}]
Let $f: V \to \ol{V}$ be the composite of the contraction of $U$ and the contractions of all subsequent contractible components of $\Supp  T_1$ and let $\ol{D}:= f_*(D)$. Then $\ol{D}$ is an SNC-divisor. 
\item[{\rm (3)}]
For any positive integer $n$, $\ol{P}_n (\ol{V}\setminus \Supp \ol{D}) = \ol{P}_n(S)$. In particular, $\lkd{\ol{V}\setminus \Supp \ol{D}} = \lkd{S}$. 
\end{itemize}
\end{lem}

\begin{proof}
The assertions follow from the arguments in \cite[Capter 2, Section 3]{OAS}. For the reader's convenience, we give an outline of its proof. 
By the assumption, there exists an irreducible curve $U$ such that $U \cdot (K_V + D^{\#}) < 0$ and the intersection matrix of $U + \Bk(D)$ is negative definite. It is clear that $U^2 < 0$. From now on, we prove the assertion (1). The assertion (2) follows from the assertion (1) and  \cite[(2) of Lemma 2.3.7.1 (p.\ 97)]{OAS}. The assertion (3) is clear from the construction of $f$ and $U \cdot D = 1$. 

Assume that $U$ is not contained in $\Supp D$. Then $U \cdot D > 0$ since $S$ is affine and hence $U \cdot K_V < 0$. In particular, $U$ is a $(-1)$-curve. Let $Z_1, \ldots, Z_n$ be all irreducible components of $D$ meeting $U$. By \cite[Lemma 2.3.6.3 (p.\ 96)]{OAS}, we know that $U \cdot Z_i = 1$  ($i =1, \ldots, n$), $n \leq 2$ and $Z_1 \cdot Z_2 = 0$ if $n = 2$. The coefficient of $Z_i$ ($i =1, \ldots, n$) in $D^{\#}$ is less than one since $U \cdot (K_V + D^{\#}) < 0$. If $n = 1$, then the assertion follows. 

We consider the case $n = 2$ and let $T_i$ ($i = 1,2$) be the maximal admissible rational twig in $D$ containing $Z_i$.
Let $g: V \to V'$ be the composite of the contraction of $U$ and the contraction of all subsequence contractible components of $T_1 + T_2$ and set $D' = g_*(D)$. It follows from \cite[(2) of Lemma 2.3.7.1 (p.\ 97)]{OAS} that the divisor $D'$ is an SNC-divisor. By  \cite[(3) of Lemma 2.3.7.1 (p.\ 97)]{OAS}, we know that, if $g_*(U + T_1 + T_2) \not= 0$, then $g_*(U + T_1 + T_2)$ is an admissible rational twig in $D'$, which is a contradiction. So $g_*(U+T_1+T_2) = 0$. Let $D_1$ and $D_2$ be the irreducible component of $\Supp (D-(T_1 + T_2)) $ meeting $T_1$ and $T_2$, respectively. 
As seen from the proof of \cite[(2) of Lemma 2.3.7.1 (p.\ 97)]{OAS} (see its first paragraph), we know that, if $D_1 = D_2$ (resp.\ $D_1 \not= D_2$), then $g_*(D_1)$ is smooth (resp.\ $g_*(D_1) \cdot g_*(D_2) = D_1 \cdot D_2)$. This contradicts $g_*(U +T_1+T_2) = 0$. 
Therefore, the case $n = 2$ does not take place. 

Assume that $U$ is a component of $\Supp D$. Then $U$ is not an irreducible component of $\Supp (\Bk(D))$ and so the coefficient of $U$ in $D^{\#}$ equals one. Since $U \cdot (K_V + U) \leq U \cdot (K_V + D^{\#}) < 0$, $U$ is a smooth rational curve. 
Let $Z'_1, \ldots, Z'_n$ be all irreducible components of $D-U$ meeting $U$. Since the coefficient of $U$ in $D^{\#}$ equals one, the coefficient $\alpha_i$ of $Z'_i$ in $D^{\#}$ $\geq \frac{1}{2}$. Further, if $U \cdot Z'_i \geq 2$, then $\alpha_i = 1$. Since $U \cdot (K_V + D^{\#} ) < 0$ and $U \cdot (K_V + U) = -2$, we have
$$
\sum_{i=1}^n \alpha_i (U \cdot Z'_i) < 2.
$$
So $n \leq 3$. Assume that $U^2 \leq -2$. If $n = 3$, then $\alpha_1, \alpha_2, \alpha_3 < 1$ and hence $\Supp D$ consists of $U$ and three admissible rational twigs. This is a contradiction because the intersection matrix of $D$ is then negative definite. If $n = 2$ or $n=1$, then $U$ is a component of an admissible rational twig in $D$, a contradiction. Therefore, $U^2 = -1$. As seen from the argument as above, we know that $U$ is a superfluous exponential component, a contradiction. 

The first assertion of (1) is thus verified. The other assertions are clear from the argument as above.
\end{proof}

Let $S$ be a smooth affine surface and $(V,D)$ its SNC-completion. Then an almost minimal model $(\tilde{V}, \tilde{D})$ of $(V,D)$ together with a birational morphism $\mu: V \to \tilde{V}$ is obtained as a composite of the following operations (cf.\ \cite[Chapter 2, 3.11 (p.\ 107)]{OAS}):
\begin{itemize}
\item[(1)]
Contract all possible superfluous exceptional components of $D$.
\item[(2)]
If there is no superfluous exceptional component of $D$, then calculate $D^{\#}$ and $\Bk(D) = D-D^{\#}$. 
\item[(3)]
Find a $(-1)$-curve $U$ such that $U \not\subset \Supp D$, $U \cdot (K_V + D^{\#}) < 0$ and the intersection matrix of $U + \Bk(D)$ is negative definite. If there is none, then we are done. If there is one, consider the birational morphism $f: V \to \ol{V}$ explained in Lemma 2.2 (2) and set $\ol{D} = f_*(D)$.
\item[(4)]
Now repeat the operations (1), (2) and (3) all over again. 
\end{itemize}

Thus, we obtain the following lemma. See \cite[Theorem 2.3.11.1 (p.\ 107)]{OAS} for general open algebraic surfaces. 

%Lemma 2.3
\begin{lem}
Let $S$ be a smooth affine surface and let $(V,D)$ be an SNC-completion of $S$. Then there exists a birational morphism $\mu : V \to \tilde{V}$ onto a smooth projective surface $\tilde{V}$ such that the following conditions {\rm (1) $\sim$ (4)} are satisfied{\rm :}
\begin{enumerate}
\item[{\rm (1)}]
$\tilde{D}:= \mu_*(D)$ is an SNC-divisor and $\Supp \tilde{D}$ is connected. 
\item[{\rm (2)}]
The surface $\tilde{S} = \tilde{V} \setminus \Supp \tilde{D}$ is an affine open subset of $S$. If $\tilde{S} \not= S$, then $S \setminus \tilde{S}$ is a disjoint union of affine lines. 
\item[{\rm (3)}]
$\ol{P}_n(\tilde{S}) = \ol{P}_n(S)$ for any positive integer $n$. In particular, $\lkd{\tilde{S}} = \lkd{\tilde{S}}$.
\item[{\rm (4)}]
The pair $(\tilde{V}, \tilde{D})$ is almost minimal.
\end{enumerate}
\end{lem}

\begin{proof}
The assertion follows from Lemma 2.2.
\end{proof}

We call the pair $(\tilde{V}, \tilde{D})$ in Lemma 2.3 an {\em almost minimal model} of $(V,D)$. 

The following lemma is a part of \cite[Lemma 1.4]{K13}. 

% Lemma 2.4
\begin{lem}
Let $S$ and $(V,D)$ be the same as above. Assume further that $(V,D)$ is almost minimal. Then the following assertions hold true{\rm :}
\begin{itemize}
\item[{\rm (1)}]
$\lkd{S} \geq 0$ if and only if $K_V + D^{\#}$ is nef.
\item[{\rm (2)}]
If $\lkd{S} \geq 0$, then $K_V + D^{\#}$ is semiample. Moreover, we have the following{\rm :}
\begin{itemize}
\item[{\rm (2-1)}]
$\lkd{S} = 0$ $\Longleftrightarrow$ $K_V + D^{\#} \equiv 0$.
\item[{\rm (2-2)}]
$\lkd{S} = 1$ $\Longleftrightarrow$ $(K_V + D^{\#})^2 = 0$ and $K_V + D^{\#} \not\equiv 0$.
\item[{\rm (2-3)}]
$\lkd{S} = 2$ $\Longleftrightarrow$ $(K_V + D^{\#})^2 > 0$.
\end{itemize}
\end{itemize}
\end{lem}

\begin{proof}
See \cite[Lemma 1.4]{K13}. 
\end{proof}

Now, let $S$ and $(V,D)$ be the same as in Lemma 2.4. Assume further that $(V,D)$ is almost minimal and $\lkd{S} \geq 0$. 
From now on, we construct a strongly minimal model of $(V,D)$. See \cite[Section 1]{K13} for more details. By Lemma 2.4 (1), $K_V + D^{\#}$ is nef. In particular, $K_V + D \equiv (K_V+D^{\#}) + \Bk (D)$ gives the Zariski decomposition of $K_V+D$. 

%Lemma 2.5
\begin{lem}
Assume that there exists a $(-1)$-curve $E$ such that $E \cdot (K_V + D^{\#}) = 0$, $E \not\subset \Supp (\lfloor D^{\#} \rfloor)$ and the intersection matrix of $E+ \Bk(D)$ is negative definite. Let $\sigma: V \to W$ be a composite of the contraction of $E$ and the contractions of all subsequently contractible components of $\Supp (\Bk(D))$. Set $B:= \sigma_*(D)$. Then the following assertions hold true.
\begin{enumerate}
\item[{\rm (1)}]
$E \cdot D = E \cdot \lfloor D^{\#} \rfloor = 1$.
\item[{\rm (2)}]
The divisor $B$ is an SNC-divisor and each connected component of $\sigma(\Supp (\Bk(D)))$ is an admissible rational twig. 
\item[{\rm (3)}]
For every integer $n \geq 1$, $\ol{P}_n(V \setminus \Supp D) = \ol{P}_n(W \setminus \Supp B)$. 
\item[{\rm (4)}]
$K_V + D^{\#} = \sigma^*(K_W + B^{\#})$. In particular, the pair $(W,B)$ is an almost minimal SNC-pair with $\lkd{W\setminus \Supp B} = \lkd{V \setminus \Supp D} (\geq 0)$.
\end{enumerate}
\end{lem}

\begin{proof}
See \cite[Lemma 1.5]{K13}, where we note that $\Supp B$ is connected and that $B$ is neither an admissible rational rod nor an admissible rational fork since $B = \sigma_*(D)$ is a big divisor. 
\end{proof}

By using the argument as above and Lemmas 2.3 and 2.5, we have the following result. 

% Lemma 2.6
\begin{lem}
Let $S$ be a smooth affine surface of $\lkd{S} \geq 0$ and let $(V,D)$ be its SNC-completion. Then there exists a birational morphism $f: V \to \tilde{V}$ onto a smooth projective surface $\tilde{V}$ such that the following conditions are satisfied{\rm :}
\begin{enumerate}
\item[{\rm (1)}]
Set $\tilde{D}:= f_*(D)$. Then $(\tilde{V}, \tilde{D})$ is an almost minimal SNC-pair such that $\ol{P}_n(\tilde{V}\setminus \Supp \tilde{D}) = \ol{P}_n(S)$ for every $n \geq 1$. In particular, $\lkd{\tilde{V}\setminus \Supp \tilde{D}} = \lkd{S}$.
\item[{\rm (2)}]
There exist no superfluous exceptional components of $\tilde{D}$.
\item[{\rm (3)}]
There exist no $(-1)$-curves $E$ such that $E\cdot (\tilde{D}^{\#} + K_{\tilde{V}}) = 0$, $E \not\subset \Supp \lfloor \tilde{D}^{\#} \rfloor$ and the intersection matrix of $E+ \Bk(\tilde{D})$ is negative definite.
\item[{\rm (4)}]
The surface $\tilde{S} = \tilde{V} \setminus \Supp \tilde{D}$ is an affine open subset of $S$. If $\tilde{S} \not= S$, then $S \setminus \tilde{S}$ is a disjoint union of affine lines. 
\end{enumerate}
\end{lem}

In Lemma 2.6, we call the pair $(\tilde{V}, \tilde{D})$ a {\em strongly minimal model} of $(V,D)$. The SNC-pair $(V,D)$ with $\lkd{V \setminus \Supp D} \geq 0$ is said to be {\em strongly minimal} if $(V,D)$ becomes a strongly minimal model of itself.
\smallskip

We recall the structure theorem for smooth affine surfaces of $\ol{\kappa} = 1$, which is a special case of \cite[Theorem 2.1]{K13}. See also \cite[Theorem 2.3]{Kaw78} in the case where $\ch (k) = 0$. A detailed proof of \cite[Theorem 2.3]{Kaw78} is given in \cite[Chapter 2, Sections 2$\sim$4]{LNM857}. For smooth affine surfaces, we can shorten \cite[Theorem 2.1]{K13}. 

%Theorem 2.7
\begin{thm}
Let $S$ be a smooth affine surface of $\lkd{S} = 1$ and $(V,D)$ its SNC-completion. Assume that $(V,D)$ is strongly minimal. 
Let $h : V \to W$ be a successive contraction of all $(-1)$-curves $E$ such that the intersection of $E$ and the image of $K_V + D^{\#}$ equals zero. Set $C:= h_*(D^{\#})$. The following assertions hold true.

\begin{enumerate}
\item[{\rm (1)}]
For a sufficiently large integer $n$, the complete linear system $|n(K_V + D^{\#})|$ defines a $\BP^1$-fibration $\Phi: V \to B$ from $V$ onto a smooth projective curve $B$ such that $\pi:= \Phi \circ h^{-1} : W \to B$ is a relatively minimal model of $\Phi$ and that $F\cdot C = F\cdot \lfloor C \rfloor = 2$, where $F$ is a general fiber of $\Phi$.
\item[{\rm (2)}]
We set as $C = H + \sum_i d_i F_i$, where $H$ is the sum of the horizontal components of $C$ and the $F_i$'s are fibers of $\pi$. Then $H$ is an SNC-divisor and consists of either two sections or an irreducible $2$-section of $\pi$. 
\item[{\rm (3)}]
The divisor $K_W + C$ can be expressed as follows{\rm :}
$$
K_W + C = \pi^*(K_B+\delta) + \sum_i d_i F_i,
$$
where $\delta$ is a divisor on $B$ such that $t:= \deg \delta$ equals $H_1\cdot H_2$ {\rm (}resp.\ one half of the number of the branch points of $\pi|_H$, $1-g(B)${\rm )} if $H=H_1+H_2$ with sections $H_1$ and $H_2$ {\rm (}resp.\ $H$ is irreducible and $\pi|_H$ is separable, $H$ is irreducible and $\pi|_H$ is not separable{\rm )} and 
$$ d_i =\left\{\begin{array}{ll}
                         \frac{1}{2} \left(1-\frac{1}{m_i}\right) & \mbox{if}\ \# (F_i \cap H) = 1, \\
                         1-\frac{1}{m_i} & \mbox{if}\ \# (F_i \cap H) = 2, 
                        \end{array}\right. $$
where $m_i$ is a positive integer or $+\infty$.
\item[{\rm (4)}]
With the same notations as in {\rm (2)} and {\rm (3)}, the following assertions hold.
\begin{enumerate}
\item[{\rm (i)}]
If $\# (F_i \cap H) = 2$ and $0 < d_i < 1$, then $\Supp (D) \cap \Supp (h^*(F_i))$ is not connected.
\item[{\rm (ii)}]
If $\# (F_i \cap H) = 1$ and $d_i > 0$, then $d_i = \frac{1}{2}$.
\end{enumerate}
\end{enumerate}
\end{thm}

\begin{proof}
\noindent
We infer from \cite[Theorem 2.1]{K13} that,  for a sufficiently large integer $n$, the complete linear system $|n(K_V + D^{\#})|$ defines a fibration $\Phi: V \to B$ from $V$ onto a smooth projective curve $B$ such that $\Phi$ is an elliptic fibration, a quasi-elliptic fibration or a $\BP^1$-fibration, that $\pi:= \Phi \circ h^{-1} : W \to B$ is a relatively minimal model of $\Phi$ and that $F\cdot C = F\cdot \lfloor C \rfloor = 2$ {\rm (}resp.\ $F\cdot C = F\cdot \lfloor C \rfloor = 0${\rm )} if $p_a(F) = 0$ {\rm (}resp.\ $p_a(F) = 1${\rm )}, where $F$ is a general fiber of $\Phi$. Since $S$ is affine, $S$ contains no complete curves. Hence $\Phi$ is a $\BP^1$-firation. 
By \cite[Lemma 2.3]{K13}, the divisor $H$ is an SNC-divisor. 
This proves the assertion (1). The assertions (2) and (3) then follow from \cite[Theorem 2.1 (II)]{K13}. The assertion (4) follows from \cite[Remark 2.6]{K13}, where we note that $\Supp D$ is connected since $S$ is affine. 
\end{proof}

%Remark2.8
\begin{remark}
Let the notations and assumptions be the same as in Theorem {\rm 2.7}. Set $\A^1_* = \A^1 \setminus \{ 0 \}$. Then $\Phi|_{S}$ becomes an untwisted $\A^1_*$-fibration {\rm (}resp.\ a twisted $\A^1_*$-fibration, an $\A^1$-fibration{\rm )} if $H = H_1+H_2$ is reducible {\rm (}resp.\ $H$ is irreducible and $\pi|_{H}$ is separable, $H$ is irreducible and $\pi|_{H}$ is not separable{\rm )}. 
\end{remark}

\section{Proof of Theorem 1.1}

Let $S$ be a smooth affine surface of $\lkd{S} = 1$ and let $(V,D)$ be its SNC-completion. 
First of all, we prove that $| \lfloor m (K_V + D)^{+} \rfloor |$ induces a $\BP^1$-fibration for $m \geq 8$. 

There exists a birational morphism $f: V \to \tilde{V}$ onto a smooth projective surface $\tilde{V}$ such that $(\tilde{V}, \tilde{D})$, where $\tilde{D} = f_*(D)$, is a strongly minimal model of $(V,D)$. Then $(K_{\tilde{V}} + \tilde{D}^{\#}) + \Bk (\tilde{D})$ gives rise to the Zariski decomposition of $K_{\tilde{V}} + \tilde{D}$ since $K_{\tilde{V}} + \tilde{D}^{\#}$ is nef and $(K_{\tilde{V}} + \tilde{D}^{\#}) \cdot \Bk (\tilde{D}) = 0$. By Theorem 2.7 (1), we know that, for sufficiently large $n > 0$, $|n(K_{\tilde{V}} + \tilde{D}^{\#})|$ gives rise to a $\BP^1$-fibration on $\tilde{V}$ over a smooth projective curve $B$. Since 
$$
h^0 (V, \lfloor m (K_{\tilde{V}} + \tilde{D}^{\#}) \rfloor) = h^0(V, m(K_{\tilde{V}} + \tilde{D})) = h^0 (V, m(K_V + D)) = h^0 (V, \lfloor m (K_V+D)^+ \rfloor)
$$ for any positive integer $m$, we know that, for a positive integer $n$, $|\lfloor n(K_V+D)^+ \rfloor|$ induces a $\BP^1$-fibration on $V$ if and only if so does $| \lfloor n (K_{\tilde{V}} + \tilde{D}^{\#}) \rfloor|$ on $\tilde{V}$. Therefore, in order to prove Theorem 1.1, we may assume that $(V,D)$ is strongly minimal.

From now on, we use the same notations as in Theorem 2.7. 
For a positive integer $n$, $|\lfloor n(K_V+D)^+ \rfloor|$ induces a $\BP^1$-fibration on $V$ if and only if so does $| \lfloor n(K_{W} + C) \rfloor|$ on $W$ because $K_W + C = h_*(K_V+D^{\#})$.  
We set 
$$
\delta_m := m(K_B+\delta)+ \sum_i \lfloor md_i \rfloor \pi(F_i)
$$ for a positive integer $m$. 
Here $\lfloor r \rfloor$ means the integral part of a real number $r$. 
By $\lkd{S} = 1$, we have
$$
\deg (K_B + \delta) + \sum_i d_i > 0. \eqno{(3.1)}
$$
We need to find the least integer $M$ such that, for any $m \geq M$, 
$$
\deg \delta_m = m (2g(B) -2 + t) + \sum_i \lfloor md_i \rfloor \geq 2g(B) + 1 \eqno{(3.2)}
$$
holds. 
Let $s$ be the non-negative integer satisfying $\sum_i d_i = \sum_{i = 1}^s d_i$, where $d_i > 0$ if $s > 0$. 
Further, we assume that
$$
d_1 \geq d_2 \geq \cdots \geq d_s \eqno{(3.3)}
$$
if $s > 0$. 
We consider the following cases separately.
\medskip

\noindent
\textbf{Case 1: }$t\geq 3$. For any $m \geq 1$, $\deg \delta_m \geq 2m g(B) + m(t-2) \geq m (2g(B) + 1)$. So, if $m \geq 1$, (3.2) holds.
\medskip

\noindent
\textbf{Case 2:} $g (B) \geq 2$ (and $t \leq 2$). We consider the following subcases separately.
\smallskip

\noindent
\textbf{2-1:} $t \geq 0$. Then $\deg \delta_m \geq m (2 g(B) -2 ) = 2m (g(B)  - 1)$. Hence, if $m \geq 3$, (3.2) holds.
\smallskip

\noindent
\textbf{2-2:} $t < 0$. By Theorem 2.7 (2), $H$ is irreducible, $\pi|_{H}: H \to B$ is a purely inseparable double covering  and $t = 1 - g(B)$.
So $\deg \delta_m \geq m (g(B)-1) $. Hence, if $\displaystyle m \geq 2 + \frac{3}{g(B) - 1}$, (3.2) holds. In particular, if $m \geq 5$, (3.2)  holds.
\medskip

\noindent
\textbf{Case 3:} $g(B) = 1$ and $1 \leq t \leq 2$. We consider the following subcases separately.
\smallskip

\noindent
\textbf{3-1:} $t = 2$. Then $\deg \delta_m \geq 2m$. So, if $m \geq 2$, (3.2) holds.
\smallskip

\noindent
\textbf{3-2:} $t = 1$. Then $\deg \delta_m \geq m$. So, if $m \geq 3$, (3.2) holds.
\medskip

\noindent
\textbf{Case 4:} $g(B) = 1$ and $t \leq 0$. Then $t = 0$ by Theorem 2.7 (2). By (3.1), we have $s > 0$. We consider the following subcases separately. 
\smallskip

\noindent
\textbf{4-1:} $H = H_1 + H_2$, where $H_1$ and $H_2$ are sections of $\pi$. 
Since $t = 0$, $H_1 \cap H_2 = \emptyset$. 
Then $\displaystyle d_i = 1 - \frac{1}{m_i}$ for $i = 1, \ldots, s$, where $m_i$ is a positive integer or $+\infty$. Since $\Supp D$ is connected, we infer from Theorem 2.7 (4) (i) that $d_i = 1$ for some $i$. By (3.3), $d_1 = 1$. Then we have
$$
\deg \delta_m = \sum_i \lfloor md_i \rfloor  \geq \lfloor md_1 \rfloor = m.
$$
So, if $m \geq 3$, (3.2) holds.
\smallskip

\noindent
\textbf{4-2:} $H$ is irreducible and $\pi|_{H}$ is separable. By Theorem 2.7 (2) and (3), $H$ is smooth and $\pi|_{H}$ is an \'{e}tale double covering. Further, $\displaystyle d_i = 1 - \frac{1}{m_i}$ for $i = 1, \ldots, s$, where $m_i$ is a positive integer or $+\infty$. In particular, $d_i \geq \frac{1}{2}$ for $i = 1, \ldots, s$. Then we have
$$
\deg \delta_m = \sum_i \lfloor md_i \rfloor  \geq \lfloor md_1 \rfloor \geq \lfloor \frac{m}{2} \rfloor.
$$
So, if $m \geq 6$, (3.2) holds. 
\smallskip

\noindent
\textbf{4-3:} $H$ is irreducible and $\pi|_{H}$ is not separable. In this case, $\ch (k) = 2$ and $\pi|_{H}$ is purely inseparable. So $\# F \cap H = 1$ for every fiber $F$ of $\pi$. Since $\Supp D$ is connected, we infer from Theorem 2.7 (4) (ii) that $d_i = \frac12$ for $i = 1, \ldots, s$. Then we have
$$
\deg \delta_m = \sum_i \lfloor md_i \rfloor  \geq \lfloor md_1 \rfloor \geq \lfloor \frac{m}{2} \rfloor.
$$
So, if $m \geq 6$, (3.2) holds. 
\medskip

\noindent
\textbf{Case 5:} $g(B) = 0$ and $t = 2$. By (3.1), $s > 0$. We consider the following subcases separately.
\smallskip

\noindent
\textbf{5-1:} $H = H_1 + H_2$, where $H_1$ and $H_2$ are sections of $\pi$. As seen from Case 2 in the proof of Claim 2 in the proof of \cite[Lemma 2.5]{K13}, we know that $\# F_i \cap H = 2$ for $i = 1,\ldots, s$. By Theorem 2.7 (3), $\displaystyle d_i = 1 - \frac{1}{m_i}$, where $m_i \geq 2$ or $m_i = +\infty$ for $i = 1, \ldots, s$. In particular, $d_i \geq \frac{1}{2}$ for $i = 1, \ldots, s$. Then we have
$$
\deg \delta_m = \sum_{i=1}^s \lfloor md_i \rfloor  \geq  \lfloor \frac{m}{2}  \rfloor .
$$
So, if $m \geq 2$, (3.2) holds.
\smallskip

\noindent
\textbf{5-2:} $H$ is irreducible and $\pi|_{H}$ is separable. Since $t = 2$, $\pi|_{H}$ has $4$ branch points. So $g(H) = 1$, i.e., $H$ is an elliptic curve. We consider the following subcases separately.
\smallskip

\noindent
\textbf{5-2-1:} $\exists i \in \{ 1, \ldots, s \}$ s.t.\ $\# (F_i \cap H) = 2$. We may assume that $\# (F_1 \cap H) = 2$. Then $d_1 \geq \frac{1}{2}$
 by Theorem 2.7 (3). By the argument as in 5-1, we conclude that, if $m \geq 2$, (3.2) holds.
\smallskip

\noindent
\textbf{5-2-2:} $\forall i \in \{ 1, \ldots, s \}$, $\# (F_i \cap H) = 1$. By Theorem 2.7 (4) (ii), $d_i = \frac{1}{2}$ for $i = 1, \ldots, s$. By the argument as in 5-1, we conclude that, if $m \geq 2$, (3.2) holds.
\smallskip

\noindent
\textbf{5-3:} $H$ is irreducible and $\pi|_{H}$ is not separable. By Theorem 2.7 (3), $t = 1 - g(B) = 1 \not= 2$, a contradiction.
\medskip

\noindent
\textbf{Case 6:} $g(B) = 0$ and $t = 1$. By (3.1), $s \geq 2$. We consider the following subcases separately.
\smallskip

\noindent
\textbf{6-1:} $H = H_1 + H_2$, where $H_1$ and $H_2$ are sections of $\pi$. Then $H_1 \cdot H_2 = 1$. By Case 2 in the proof of \cite[Lemma 2.5]{K13}, we know that $\displaystyle d_i = 1 - \frac{1}{m_i}$ for some $m_i \in \Z_{\geq 2} \cup \{ +\infty \}$ for $i = 1, \ldots, s$. In particular, $d_i \geq \frac{1}{2}$ for $i = 1, \ldots, s$. We consider the following subcases separately.
\smallskip

\noindent
\textbf{6-1-1:} $s \geq 3$. Then we have
$$
\deg \delta_m = -m + \sum_{i=1}^s \lfloor m d_i \rfloor \geq -m + 3 \lfloor \frac{m}{2} \rfloor. 
$$
So, if $m \geq 4$, then (3.2) holds. 
\smallskip

\noindent
\textbf{6-1-2:} $s = 2$. By (3.1), we have $\displaystyle d_1 + d_2 = 2 - \left( \frac{1}{m_1} + \frac{1}{m_2} \right) > 1$. By (3.3), $2 \leq m_2 \leq m_1$ and $m_1 \geq 3$. Then we have
$$
\deg \delta_m = -m +  \lfloor m d_1 \rfloor +  \lfloor m d_2 \rfloor \geq -m + \lfloor \frac{2m}{3} \rfloor  + \lfloor \frac{m}{2} \rfloor. 
$$
Therefore, if $m \geq 8$, then (3.2) holds.
\smallskip

\noindent
\textbf{6-2:} $H$ is irreducible and $\pi|_{H}: H \to B \cong \BP^1$ is separable. Since $t = 1$, $g(H) = 0$. Let $P_1$, $P_2$ be the two branch points of $\pi|_{H}$. By Theorem 2.7 (4), we know that if $F_i = \pi^{-1}(P_i)$ for some $i \in \{ 1,2 \}$, then $d_i = \frac{1}{2}$. Set $r = \# \{ i \ | \ 1 \leq i \leq s, F_i = \pi^{-1}(P_j)$ for some $j \in \{ 1, 2 \} \}$. Then $r \in \{ 0, 1, 2 \}$. We consider the following subcases separately.
\smallskip

\noindent
\textbf{6-2-1:} $r = 0$. The numerical conditions on $t$ and $d_i$'s are the same as in 6-1. So we know that:
\begin{itemize}
\item
If $s \geq 3$, then (3.2) holds when $m \geq 4$.
\item
If $s = 2$, then (3.2) holds when $m \geq 8$.
\end{itemize}
Therefore, if $m \geq 8$, then (3.2) holds.
\smallskip

\noindent
\textbf{6-2-2:} $r = 1$. We may assume that $F_s  = \pi^{-1}(P_1)$. Then $d_s = \frac{1}{2}$. In this subcase, the numerical conditions on $t$ and $d_i$'s are the same as in 6-1. Therefore, the same conclusion as that in 6-2-1 holds.
\smallskip

\noindent
\textbf{6-2-3:} $r = 2$. By Theorem 2.7 (3) and (4) and (3.3), we may assume that $F_{s-1} = \pi^{-1}(P_1)$ and $F_s = \pi^{-1}(P_2)$. Then $d_{s-1} = d_s = \frac{1}{2}$. By (3.1), $s \geq 3$. 
By the argument as in 6-1-1, we know that (3.2) holds when $m \geq 4$. 
\smallskip

\noindent
\textbf{6-3:} $H$ is irreducible and $\pi|_{H}: H \to B \cong \BP^1$ is not separable. Then $\ch (k) = 2$ and $\pi|_{H}$ is a purely inseparable double covering. For every $i = 1,2, \ldots, s$, $\# F_i \cap H = 1$. So $\displaystyle d_i = \frac{1}{2}$ for every $i$ by Theorem 2.7 (4). By (3.1), $s \geq 3$. By the argument as in 6-1-1, we know that (3.2) holds when $m \geq 4$. 
\medskip

\noindent
\textbf{Case 7:} $g(B) = 0$ and $t \leq 0$. By the condition on $t$, $t = 0$. By (3.1), $s \geq 3$. We consider the following subcases separately.
\smallskip

\noindent
\textbf{7-1:} $H = H_1 + H_2$, where $H_1$ and $H_2$ are sections of $\pi$. Then $H_1 \cdot H_2 = 0$. By Theorem 2.7 (3), we know that   $\displaystyle d_i = 1 - \frac{1}{m_i}$ for some $m_i \in \Z_{\geq 2} \cup \{ +\infty \}$ for $i = 1, \ldots, s$. Since $\Supp D$ is connected and $H_1 \cdot H_2 = 0$, we infer from Theorem 2.7 (4) (i) that $d_i = 1$ for some $i = 1, \ldots, s$. By (3.3), $d_1 = 1$. Then $\sum_{i = 2}^s d_i > 1$. 
We consider the following subcases separately.
\smallskip

\noindent
\textbf{7-1-1:} $s \geq 4$. Then we have
$$
\deg \delta_m \geq -m + 3 \lfloor \frac{m}{2} \rfloor. 
$$
So, if $m \geq 4$, then (3.2) holds. 
\smallskip

\noindent
\textbf{7-1-2:} $s = 3$. By (3.1), we have $\displaystyle d_2 + d_3 = 2 - \left( \frac{1}{m_2} + \frac{1}{m_3} \right) > 1$. By (3.3), $2 \leq m_3 \leq m_2$ and $m_2 \geq 3$. Then we have
$$
\deg \delta_m \geq -m + \lfloor \frac{2m}{3} \rfloor  + \lfloor \frac{m}{2} \rfloor. 
$$
Therefore, if $m \geq 8$, then (3.2) holds.
\smallskip

\noindent
\textbf{7-2:} $H$ is irreducible. Since $t = 0$, we infer from Theorem 2.7 (3) that $\pi|_H : H \to B \cong \BP^1$ is a separable double covering. Here we note that $H$ is smooth and $\pi|_H$ is unramified. However, this is not the case because $B \cong \BP^1$.  
 Therefore, this case does not take place. 
\medskip

Therefore, we know that, in every case, if $m \geq 8$, then (3.2) holds. This proves the first assertion of Theorem 1.1.

%Remark 3.1
\begin{remark}
In {\rm 2-2} of Case {\rm 2}, we can prove $s \geq 1$ by using results in \cite[\S 2]{M82}. However, we do not use this remark because it does not give better bound. 
\end{remark}

Next, we prove the last assertion of Theorem 1.1. We construct a smooth affine surface $S$ with an SNC-completion $(V,D)$ such that $\lfloor 7 (K_V + D)^{+} \rfloor$ does not induce a $\BP^1$-fibration on $V$. The notions given below are synchronized with those in Sections 2.

%Example 3.2
\begin{example}
{\rm 
Let $W = \BP^1 \times \BP^1$ and let $\pi: W \to \BP^1$ be the first projection, which gives a structure of $\BP^1$-bundle of $W$ over $\BP^1$. Let $F_1, F_2, F_3$ be three distinct fibers of $\pi$ and let $H_1, H_2$ be two distinct fibers of the second projection. Then $H_1, H_2$ are sections of $\pi$. We set $d_1 = 1$, $d_2 = \frac23$, $d_3 = \frac12$ and set $C = H_1 + H_2 + d_1 F_1 + d_2 F_2 + d_3 F_3 = H_1 + H_2 + F_1 + \frac23 F_2 + \frac12 F_3$. Let $h : V \to W$ be a composition of blowing-ups at $P = F_2 \cap H_1$ and $Q = F_3 \cap H_2$ and their infinitely near points such that $h^*(F_2)_{\red} = [3,1,2,2]$ and $h^*(F_3)_{\red} = [2,1,2]$. Here, we denote the proper transformations of $H_1$, $H_2$ and $F_1$ on $V$ by the same letters. Let $E_i$ ($i=2,3$) be the unique $(-1)$-curve in $\Supp h^*(F_i)$ and set $D = h^*(H_1+H_2+F_1+F_2+F_3)_{\red} - (E_2 + E_3)$. The configuration of $H_1 + H_2 + F_1 + h^*(F_2)_{\red} + h^*(F_3)_{\red}$ on $V$ is given in Figure 1. Here, we have $h^*(F_2) = D_1 + 3E_2 + 2D_3 + D_2$, $h^*(F_3) = D_4 + 2E_3 + D_5$, $D_1^2 = -3$, $D_2^2 = \cdots = D_5^2 = -2$. Set $S = V \setminus \Supp D$. We note that $\pi \circ h |_{S} : S \to \A^1$ is a untwisted $\A^1_*$-fibration over $\A^1$. 
}
\end{example}

\raisebox{-55mm}{
\setlength{\unitlength}{1mm}
\begin{picture}(130,50)(0,0)
\put(60,0){Figure 1.}
\put(40,10){\line(1,0){50}}
\put(35,9){\scriptsize{$H_2$}}
\put(91,9){\scriptsize{$-1$}}
\put(45,8){\line(0,1){49}}
\put(40,33){\scriptsize{$F_1$}}
\put(46,33){\scriptsize{$0$}}
\put(40,55){\line(1,0){50}}
\put(35,54){\scriptsize{$H_1$}}
\put(91,54){\scriptsize{$-1$}}
\put(58,25){\line(1,-2){9}}
\put(59,12){\scriptsize{$D_2$}}
\put(64,14){\scriptsize{$-2$}}
\put(58,17){\line(1,2){9}}
\put(59,28){\scriptsize{$D_3$}}
\put(64,25){\scriptsize{$-2$}}
\multiput(58,42)(1,-1){12}{\circle*{0.2}}
\put(63,38){\scriptsize{$E_2$}}
\put(58,37){\line(1,2){11}}
\put(59,49){\scriptsize{$D_1$}}
\put(64,47){\scriptsize{$-3$}}
\put(75,27){\line(1,-2){10}}
\put(77,12){\scriptsize{$D_5$}}
\put(82,14){\scriptsize{$-2$}}
\put(75,37){\line(1,2){10}}
\put(77,50){\scriptsize{$D_4$}}
\put(82,48){\scriptsize{$-2$}}
\multiput(77,20)(0,1){26}{\circle*{0.2}}
\put(78,30){\scriptsize{$E_3$}}
\end{picture}}
\medskip

%Lemma 3.3
\begin{lem}
With the same notations and assumptions as above, the following assertions hold.
\begin{itemize}
\item[(1)]
$S$ is a smooth affine surface.
\item[(2)]
$S$ admits a hyperbolic $\G_m$-action whose quotient morphism is $ \pi \circ h |_{S}$. {\rm (}See \cite[\S 1]{R92} for the definition of a hyperbolic $\G_m$-action.{\rm )}
\end{itemize}
\end{lem}

\begin{proof}
We first prove the assertion (2). In fact, the assertion (2) easily follows from results of \cite{R92}. We give an outline of the proof. Since $H_1$ and $H_2$ (on $W$) are sections of $\pi$ and $H_1 \cap H_2 = \emptyset$, there exists an effective $\G_m$-action $\sigma$ on $W$ such that $W^{\sigma} = H_1 \cup H_2$ and $\pi$ becomes the quotient morphism. Since $h : V \to W$ is a composition of blowing-ups over points on $H_1 \cup H_2$, which are fixed points of $\sigma$, and $h^*(F_j)_{\red}$ ($j = 2,3$) is a chain of $\BP^1$'s, there exists uniquely a $\G_m$-action $\tilde{\sigma}$ on $V$ extending the action on $V \setminus (H_1 \cup H_2)$. Each irreducible component of $h^*(F_j)_{\red}$ ($j = 2,3$) is invariant under the action $\tilde{\sigma}$. So the action $\tilde{\sigma}$ induces a hyperbolic $\G_m$-action  on $S = V \setminus \Supp D$. This proves the assertion (2).

The assertion (1) can be verified easily. Here, we prove it using results of \cite{R92}. By the construction of $(V,D)$, we know that $(V,D)$ is a $\G_m$-pair in the sense of \cite[Definition (3.2)]{R92}, where $G$ is $\G_m$ in our notation. Namely, we see that:
\begin{itemize}
\item[(i)]
$V$ is a smooth projective $\G_m$-surface without elliptic fixed points.
\item[(ii)]
$\Supp D$ is an invariant connected closed curve on $V$.
\item[(iii)]
Every invariant closed curve on $V$ meets $\Supp D$.
\end{itemize}
By \cite[Theorem (4.9)]{R92}, $S = V \setminus \Supp D$ is affine. 
\end{proof}

%Lemma 3.4
\begin{lem}
With the same notations and assumptions as above, the following assertions hold.
\begin{itemize}
\item[(1)] 
$\lkd{S} = 1$.
\item[(2)]
$D^{\#} = D - \Bk (D) = H_1 + H_2 + F_1 + \frac{2}{3} (D_1 + D_2) + \frac{1}{3} D_3 + \frac{1}{2} (D_4 + D_5)$ and the $\Q$-divisor $K_V + D^{\#}$ is the nef part of the Zariski decomposition of $K_V + D$. 
\item[(3)]
$\lfloor 7(K_V + D^{\#}) \rfloor$ does not induce a $\BP^1$-fibration on $V$. 
\end{itemize}
\end{lem}

\begin{proof}
By the definition of $C$ and $K_W \sim -H_1 - H_2 - 2F_1$, we have
$$
K_W + C \sim -F_1 + \frac{2}{3} F_2 + \frac{1}{2} F_3 \sim_{\Q} \frac{1}{6} F_1,  
$$
where $\sim_{\Q}$ means the $\Q$-linear equivalence relation on $\Q$-divisors. Since $F_1$ is a smooth rational curve with $F_1^2 = 0$, we know that $K_W + C$ is nef and $\kappa(W, K_W+C) = 1$, where $\kappa(W, K_W+C)$ denotes the Iitaka dimension of $K_W+C$. 

By the definition of $D^{\#} = D - \Bk (D)$ in \S 2, we have 
$$
D^{\#} = H_1 + H_2 + F_1 + \frac{2}{3} (D_1 + D_2) + \frac{1}{3} D_3 + \frac{1}{2} (D_4 + D_5).
$$
By the construction of $h$, we see that
$$
K_V \sim -2F_1 - H_1 - H_2 + D_3 +2 E_2 + E_3. 
$$
So, 
$$
K_V + D^{\#} = -F_1+  \frac{2}{3} D_1 + \frac{2}{3} D_2 + \frac{4}{3} D_3 + 2 E_2 + \frac{1}{2} D_4 + \frac{1}{2} D_5 + E_3.   \eqno{(3.4)}
$$
In particular,
$$
K_V + D^{\#} \sim_{\Q} h^*(K_W + C).
$$
This implies that $K_V + D^{\#}$ is nef and $\kappa(V, K_V + D^{\#}) = 1$. This proves the assertions (1) and (2). 

We prove the assertion (3). By (3.4) and 
$$
\frac{2}{3} D_1 + \frac{2}{3} D_2 + \frac{4}{3} D_3 + 2 E_2 \sim_{\Q} \frac{2}{3} h^*(F_2), \quad \frac{1}{2} D_4 + \frac{1}{2} D_5 + E_3 \sim_{\Q} \frac{1}{2} h^*(F_3),
$$
we have
$$
\lfloor 7 (K_V + D^{\#}) \rfloor \sim -7F_1 + \lfloor \frac{14}{3} \rfloor F_1 +   \lfloor \frac{7}{2} \rfloor F_1 = 0.
$$
Therefore, $\lfloor 7(K_V + D^{\#}) \rfloor$ does not induce a $\BP^1$-fibration on $V$. 
\end{proof}

The proof of Theorem 1.1 is thus completed. 

\section{Some remarks}

In this section, we consider Question in Section 1 for special smooth affine surfaces of $\ol{\kappa} = 1$. 

%Proposition 4.1
\begin{prop}
For the smooth affine surfaces of $\ol{\kappa} = 1$ with twisted $\A^1_*$-fibrations, we have $M = 8$ {\rm (}in Question in Section {\rm 1}{\rm )} and $8$ is best possible. 
\end{prop}

\begin{proof}
As seen from the argument as in Section 3, we know that $M = 8$. We construct a smooth affine surface $S$ of $\lkd{S} = 1$ with a twisted $\A^1_*$-fibration for which the number $8$ is best possible. 

Let $W = \BP^1 \times \BP^1$ and let $\pi: W \to \BP^1$ be the first projection. Let $F$ (resp.\ $G$) be a fiber of $\pi$ (resp.\ a fiber of the second projection of $W = \BP^1 \times \BP^1$). We take an irreducible curve $H \sim 2G + F$ such that $\pi|_{H}: H \to \BP^1$ is a separable double covering. Furthermore, we take two fibers $F_1$ and $F_2$ of $\pi$ such that $\# F_i \cap H = 2$ for $i = 1,2$. Let $P_i$ ($i=1,2$) be one of the two points $F_i \cap H$. We set $d_1 = \frac{2}{3}$, $d_2 = \frac{1}{2}$ and set $C = H + d_1 F_1 + d_2 F_2 = H + \frac{2}{3} F_1 + \frac{1}{2} F_2$. Let $h: V \to W$ be a composition of blowing-ups at $P_1$ and $P_2$ and their infinitely near points such that $h^*(F_1)_{\red} = [3,1,2,2]$ and $h^*(F_2)_{\red} = [2,1,2]$. Here, we denote the proper transform of $H$ on $V$ by the same letter. Let $E_i$ ($i=1,2$) be the unique $(-1)$-curve in $\Supp h^*(F_i)$ and set $D = h^*(H+F_1+F_2)_{\red} - (E_1+E_2)$. We have $h^*(F_1) = D_1 + 3E_1 + 2 D_3 + D_2$ and $h^*(F_2) = D_4 + 2E_2 + D_5$, where $D_1$ is a $(-3)$-curve, $D_2, \ldots, D_5$ are $(-2)$-curves, $D_1 \cdot E_1 = E_1 \cdot D_3 = D_3 \cdot D_2 = D_4 \cdot E_2 = E_2 \cdot D_5 = 1$. Set $S = V \setminus \Supp D$. Then $\pi \circ h|_{S} : S \to \BP^1$ is a twisted $\A^1_*$-fibration onto $\BP^1$. 

We prove that $S$ is affine. Let $H_i$ ($i=1,2)$ be the fiber of the second projection passing through $P_i$. Since $H \sim 2G + F$, $H_1 \not= H_2$. Further, since $H \sim F + H_1 + H_2$ and $P_i \in H_i \cap H$ for $i=1,2$, we have 
$$
h^*(H) \sim h^*(F+H_1+H_2) \sim H + D_1 + D_3 + 2E_1 + D_5 + E_2 
$$
$$
\sim h^{-1}_*(F) + h^{-1}_*(H_1) + h^{-1}_*(H_2) + D_1 + D_3 + 2 E_1 + D_5 +E_2.
$$
Namely, $D \sim h^{-1}_*(F) + h^{-1}_*(H_1) + h^{-1}_*(H_2) + D_1 + D_2 + D_3 + D_4 + D_5$. 
Since $ h^{-1}_*(F) + h^{-1}_*(H_1) + h^{-1}_*(H_2) + D_1 + \cdots + D_5$ is the divisor $D$ in Example 3.2, we know that $V \setminus \Supp ( h^{-1}_*(F) + h^{-1}_*(H_1) + h^{-1}_*(H_2) + D_1 + \cdots + D_5)$ is affine by Lemma 3.3 (1). 
Here we set $L = 6 (h^{-1}_*(F) + h^{-1}_*(H_1) + h^{-1}_*(H_2) ) + D_1 + 3 D_2 + D_3 + D_4 + D_5$. Then $L \cdot \tilde{D} > 0$ for any irreducible curve $\tilde{D}$ in $\Supp (h^{-1}_*(F) + h^{-1}_*(H_1) + h^{-1}_*(H_2) + D_1 + \cdots + D_5)$. Since $V \setminus \Supp (h^{-1}_*(F) + h^{-1}_*(H_1) + h^{-1}_*(H_2) + D_1 + \cdots + D_5)$ is affine, we have $L \cdot D' > 0$ for any irreducible curve $D'$ on $V$. So $L$ is ample by the Nakai--Moishezon criterion. (See also \cite[Proposition (4.1)]{R92}.) Since $L \sim 6 H + D_1 + 3D_2 + D_3 + D_4 + D_5$ is ample, we know that $S = V \setminus \Supp (6 H + D_1 + 3D_2 + D_3 + D_4 + D_5)$ is affine. 

Since the numerical conditions on $ h^{-1}_*(F) + h^{-1}_*(H_1) + h^{-1}_*(H_2) + D_1 + \cdots + D_5$ is the same as $D$ in Example 3.2, where we consider $F$ (resp.\ $H_1$, $H_2$) as $h^{-1}_*(F)$ (resp.\ $h^{-1}_*(H_1)$, $h^{-1}_*(H_2)$), we infer from Lemma 3.4 that:
\begin{itemize}
\item[(1)]
$\lkd{S} = 1$.
\item[(2)]
$D^{\#} = H + \frac{2}{3} (D_1 + D_2) + \frac{1}{3} D_3 + \frac{1}{2} (D_4 + D_5)$ and the $\Q$-divisor $K_V + D^{\#}$ is the nef part of the Zariski decomposition of $K_V+D$. 
\item[(3)]
$\lfloor 7 (K_V + D^{\#}) \rfloor$ does not induce a $\BP^1$-fibration on $V$.
\end{itemize}
This proves Proposition 4.1. 
\end{proof}

Finally, we prove the following result.

%Proposition 4.2
\begin{prop}
For the smooth affine surfaces of $\ol{\kappa} = 1$ with $\A^1$-fibrations, we have $M = 6$ {\rm (}in Question in Section {\rm 1}{\rm )} and $6$ is the best possible.
\end{prop}

\begin{proof}
In our setting, $\ch(k) = 2$. As seen from the arguments as in Section 3, we know that $M = 6$. We construct a smooth affine surface $S$ of $\lkd{S} = 1$ with an $\A^1$-fibration for which the number $6$ is best possible.

Let $B$ be an elliptic curve. We construct a Frobenius pair $(W,H)$ over $B$ as in \cite[2.1 and 2.2]{M82} (see also \cite[Example 2.2]{K17}). Let $\phi: B \to B$ be the absolute Frobenius morphism. Then we have an exact sequence
$$
0 \to \SO_B \to \phi_*\SO_B \to \SL \to 0,
$$
where $\phi_*\SO_B$ is a vector bundle of rank two over $B$ and $\SL$ is an invertible sheaf of $\deg \SL = 0$. 
By \cite[Lemmas 2.4 and 2.6]{M82}, $2 \SL \sim 0$. The vector bundle $\phi_*\SO_B$ defines a $\BP^1$-bundle $\pi: W:= \BP(\phi_*\SO_B) \to B$ and the surjection $\phi_*\SO_B \to \SL$ defines a section $M$. Moreover, the $\SO_B$-algebra $\phi_*\SO_B$ defines a smooth projective curve $H$ on $W$ such that $\pi|_{H}: H \to B$ is identified with $F: B \to B$ (cf.\ \cite[2.1]{M82}). The pair $(W, H)$ is called the Frobenius pair over $B$.
By \cite[Lemmas 2.4 and 2.6]{M82}, $|K_W + H| = \emptyset$ and $2(K_W + H) \sim 0$. In particular, $K_W + H \equiv 0$. Let $F_1$ be a fiber of the ruling $\pi: W \to B$. We set $d_1 = \frac{1}{2}$ and $C = H + d_1 F_1 = H + \frac{1}{2} F_1$. It is clear that $\# H \cap F_1 = 1$. Let $h : V \to W$ be a composition of blowing-ups at $P = H \cap F_1$ and its infinitely near points such that $h^*(F_1)_{\red} = [2,1,2]$. Then $h^*(F_1) = D_1 + 2E_1 + D_2$, where $D_1 = h^{-1}_*(F_1)$, $D_2$ is a $(-2)$-curve, $E_1$ is a $(-1)$-curve,  and $D_1 \cdot E_1 = E_1 \cdot D_2 = 1$. Set $D = h^{-1}_*(H) + D_1 + E_1 + D_2$. Then $D$ is an SNC-divisor. Set $S = V \setminus \Supp D = W \setminus (H \cup F_1)$. Then $\pi|_{S} : S \to B \setminus \pi(F_1)$ is an $\A^1$-fibration over $B \setminus \pi(F_1)$.
Since $H$ is a $2$-section of $\pi$ and $F_1$ is a fiber of $\pi$, a divisor $H + \alpha F_1$ is ample for sufficiently large $\alpha$ by \cite[Propositions V.2.20 (b) and V.2.21(b)]{GTM52}. So $S$ is affine. 

It is easy to see that $D^{\#} = h^{-1}_*(H) + E_1 + \frac{1}{2}(D_1 + D_2)$. Since $K_W + H \equiv 0$ and by the construction of $h$, we know that $\lkd{S} = 1$ and $K_V + D^{\#} = h^*( K_W + C)$ is the nef part of the Zariski decomposition of $K_V+D$. Finally, we have
$$
\lfloor 5 (K_V + D^{\#}) \rfloor \sim \lfloor \frac{5}{2} F_1 \rfloor = 2 F_1 = \pi^*(2 \pi (F_1)).
$$
Since $2 \pi (F_1)$ is not very ample, $\lfloor 5 (K_V + D^{\#}) \rfloor $ does not induce a $\BP^1$-fibration. This proves Proposition 4.2.
\end{proof}

\end{document}